\documentclass[a4paper,twoside]{amsart}
\usepackage{geometry}
\usepackage[british,UKenglish,USenglish,english,american]{babel}
\usepackage{newlfont}
\usepackage{amssymb}
\usepackage{amsmath}
\usepackage{amsfonts}
\usepackage{latexsym}
\usepackage{amsthm}
\usepackage{mathrsfs}
\usepackage{stmaryrd}
\usepackage[all,cmtip]{xy}
\usepackage{caption}
\usepackage{enumerate}
\usepackage{hyperref}
\usepackage{stackrel}
\usepackage{graphicx}

\usepackage{tikz-cd}

\newcommand{\bb}{\mathbb}
\newcommand{\op}{\operatorname}

\newtheorem{theoremalpha}{Theorem}

\newtheorem{Cor1}[theoremalpha]{Corollary}
\newtheorem{Def-Teo1}{Definition-Theorem}
\newtheorem{question}{Question}

\newtheorem{Teo}{Theorem}[section]
\newtheorem{Prop}[Teo]{Proposition}

\newtheorem{Lem}[Teo]{Lemma}

\theoremstyle{definition}

\newtheorem{Def}[Teo]{Definition}
\newtheorem{Prop-Def}[Teo]{Definition-Proposition}

\newtheorem{Oss}[Teo]{Remark}

\newenvironment{dimo}
         {\textit{Proof of}}
     {\hspace*{\fill}\hspace*{\fill}\mbox{$\Box$}}
     
\newcommand{\Pic}{\text{Pic}}

\newcommand{\oo}{\mathcal{O}}
\newcommand{\mt}{\mathcal}
\newcommand{\ms}{\mathscr}
\newcommand{\bigslant}[2]{{\raisebox{.2em}{$#1$}\left/\raisebox{-.2em}{$#2$}\right.}}

\begin{document}
\title[The Picard group of the universal abelian variety]{The Picard group of the universal abelian variety and the Franchetta conjecture for abelian varieties}
\author{Roberto Fringuelli and Roberto Pirisi}

\maketitle

\begin{abstract}We compute the Picard group of the universal abelian variety over the moduli stack $\ms A_{g,n}$ of principally polarized abelian varieties over $\mathbb{C}$ with a symplectic principal level $n$-structure. We then prove that over $\mathbb{C}$ the statement of the Franchetta conjecture holds in a suitable form for $\ms A_{g,n}$.
\end{abstract}

\tableofcontents

\section*{Introduction}

Consider the moduli stack $\ms M_g$ of smooth genus $g$ curves. Let $C_{\eta}$ the universal curve over the generic point $\eta$ of $\ms M_g$. The weak Franchetta conjecture says that $\text{Pic}(C_{\eta})$ is freely generated by the cotangent bundle $\omega_{C_{\eta}}$. Arbarello and Cornalba in \cite{AC87} proved it over the complex numbers. Then Mestrano \cite{Franc} and Kouvidakis \cite{Kou91} deduced over $\mathbb C$ the strong Franchetta conjecture, which says that the rational points of the Picard scheme $Pic_{C_\eta}$ are precisely the multiples of the cotangent bundle. Then Schr\"oer \cite{Schroer} proved both the conjectures over an algebraically closed field of arbitrary characteristic. At the end of \emph{loc. cit.}, Schr\"oer poses the question of whether it is possible to generalize the Franchetta Conjecture to other moduli problems.
 
In this paper we focus on the universal abelian variety $\ms X_{g,n}$ over the moduli stack $\ms A_{g,n}$ of principally polarized abelian varieties of dimension $g$ (or p.p.a.v. in short) with a symplectic principal level-$n$ structure (or level-$n$ structure in short). For the analogous of the Franchetta conjecture in this new setting, we have chosen the name of \emph{abelian Franchetta conjecture}.

First of all, observe that the universal abelian variety $\pi:\ms X_{g,n}\to\ms A_{g,n}$ comes equipped with some natural line bundles: 
\begin{itemize}
\item[-] the \emph{rigidified canonical line bundle} $\mt L_{\Lambda}$, i.e. a line bundle, trivial along the zero section, such that over any closed point $(A,\lambda,\varphi)\in\ms A_{g,n}$, the restriction of $\mt L_{\Lambda}$ along the fiber $\pi^{-1}\left((A,\lambda,\varphi)\right)$ induces twice the principal polarization $\lambda$.
\item[-] the \emph{rigidified $n$-roots line bundles}, i.e. the line bundles, trivial along the zero section, which are $n$-roots of the trivial line bundle $\oo_{\ms X_{g,n}}$. 
\end{itemize}

We can formulate the weak abelian Franchetta conjecture in terms of a description of the Picard group of a generic universal abelian variety over $\ms A_{g,n}$ (observe that if $n=1,2$ the generic point of $\ms A_{g,n}$ is stacky, so it makes little sense to speak of ``the'' generic abelian variety), and the strong one in terms of rational sections of the relative Picard scheme. %(resp. of the group of global sections of the \'etale sheaf $R^1\pi_*\bb G_m$ over $\eta$, i.e. $H^0(\eta,R^1\pi_*\bb G_m)$).

\begin{question}[Weak abelian Franchetta Conjecture]
 Is there a principally polarized abelian variety with level-$n$ structure $(A,\Lambda, \varPhi)$ over a field $K$ such that $A \rightarrow \ms X_{g,n}$ is a dominant map, and the Picard group of $A$ is freely generated by the rigidified canonical line bundle and the rigidified $n$-roots line bundles?
\end{question}

\begin{question}[Strong abelian Franchetta Conjecture]
Does every rational section of the relative Picard sheaf $Pic_{\mathscr{X}_{g,n}/\mathscr{A}_{g,n}} \rightarrow \mathscr{A}_{g,n}$ come from one of the elements above?
\end{question}

At first sight, the two question seems different. The reason is that, in general, the relative Picard group of a scheme $f:X\to S$ does not coincide with the $S$-sections of the associated Picard sheaf $\Pic_{X/S}$. However, they are isomorphic, if the scheme $X\to S$ admits a section and the structure homomorphism $\oo_S\to f_*\oo_X$ is universally an isomorphism. Since the universal abelian variety satisfies both these properties, the conjectures are equivalent. In other words, we can formulate the abelian Franchetta conjecture in terms of the set of \emph{rational relative line bundles on $\ms X_{g,n}\to\ms A_{g,n}$}: the set of the equivalence classes of line bundles on $\ms U\times_{\ms A_{g,n}}\ms X_{g,n}\to\ms U$ where $\ms U$ is an open substack of $\ms A_{g,n}$. Two line bundles are in the same class if and only if they are isomorphic along an open subset of $\ms A_{g,n}$. Observe that the tensor product induces a well-defined group structure on this set. Our main result is
\begin{theoremalpha}\label{picxg}Assume that $g\geq 4$ and $n\geq 1$. Then
$$
\bigslant{\Pic(\ms X_{g,n})}{\Pic(\ms A_{g,n})}=
\begin{cases}
\left(\bb Z/n\bb Z\right)^{2g}\oplus \bb Z[\sqrt{\mt L_{\Lambda}}] & \text{if }n \text{ is even,}\\
\left(\bb Z/n\bb Z\right)^{2g}\oplus \bb Z[\mt L_{\Lambda}] & \text{if }n \text{ is odd,}
\end{cases}
$$ 
where $\sqrt{\mt L_{\Lambda}}$ is, up to torsion, a square-root of the rigidified canonical line bundle $\mt L_{\Lambda}$ and $\left(\bb Z/n\bb Z\right)^{2g}$ is the group of rigidified $n$-roots line bundles. Moreover, the line bundle $\sqrt{\mt L_{\Lambda}}$, when it exists, can be chosen symmetric.
\end{theoremalpha}

The main difficulty in proving this theorem resides in the fact that differently from $\ms M_g$, the stack $\ms A_g$ is not generically a scheme. To solve this, we use the techniques of equivariant approximation, first introduced by Totaro, Edidin and Graham in \cite{Equiv}, \cite{Tot}.

However, a description of the entire Picard group $\Pic(\ms X_{g,n})$ is still incomplete. Ideed, while it is well known that the Picard group of $\ms A_g:=\ms{A}_{g,1}$ is freely generated by the Hodge line bundle $\op{det}\left(\pi_*\left(\Omega_{\ms X_{g}/\ms A_{g}}\right)\right)$ (see \cite[Theorem 5.4]{Pu2012}), the same is not true in general. For some results about the Picard group of $\ms A_{g,n}$ the reader can refer to \cite{Pu2012}. 

The above theorem implies directly

\begin{Cor1}\label{franc}$[$Abelian Franchetta conjecture$]$. Assume $g\geq 4$ and $n\geq 1$. The group of rational relative rigidified line bundles on $\ms X_{g,n}\to\ms A_{g,n}$ is isomorphic to
$$
\begin{array}{ll}
\left(\bb Z/n\bb Z\right)^{2g}\oplus \bb Z[\sqrt{\mt L_{\Lambda}}] & \text{if }n \text{ is even,}\\
\left(\bb Z/n\bb Z\right)^{2g}\oplus \bb Z[\mt L_{\Lambda}] & \text{if }n \text{ is odd.}
\end{array}
$$ 
where $\sqrt{\mt L_{\Lambda}}$ is, up to torsion, a square-root of the rigidified canonical line bundle $\mt L_{\Lambda}$ and $\left(\bb Z/n\bb Z\right)^{2g}$ is the group of rigidified $n$-roots line bundles. Moreover, the line bundle $\sqrt{\mt L_{\Lambda}}$, when it exists, can be chosen symmetric.
\end{Cor1}

When $g=2,3$ and $n>1$ we can prove Theorem \ref{picxg} and Corollary \ref{franc} only when $n$ is even. It still remains to find out whether the torsion free part is generated by $\mt L_{\Lambda}$ or $\sqrt{\mt L_{\Lambda}}$, when $n$ is odd.

When $n=1$, we can use the Torelli morphism to extend the result to the genus two and three cases. Let $\ms J_g$ be the universal Jacobian on $\ms M_g$ of degree $0$. We have a cartesian diagram of stacks
$$
\xymatrix{
\ms J_g\ar[r]^{\widetilde \tau_g}\ar[d]&\ms X_g\ar[d]\\
\ms M_g\ar[r]^{\tau_g} & \ms A_g
}
$$
where the map $\tau_g$ is the Torelli morphism. Observe that the Hodge line bundle on $\ms X_{g,n}$ restricts to the Hodge line bundle $\op{det}\left(\pi_*\left(\omega_{\ms M_{g,1}/\ms M_g}\right)\right)$ on $\ms M_g$. In particular the Torelli morphism induces an isomorphism of Picard groups $\Pic(\ms A_g)\cong\Pic(\ms M_g)$ for $g\geq 3$.

We will show that we have an analogous result for the universal families. More precisely

\begin{theoremalpha}\label{corxg}Assume that $g\geq 2$. Then
$$
\bigslant{\Pic(\ms X_{g})}{\Pic(\ms A_{g})}=\mathbb Z[\mt L_{\Lambda}]
$$
Furthermore, when $g\geq 3$, the morphism $\widetilde \tau_g:\ms J_g\to\ms X_g$ induces an isomorphism of Picard groups.
\end{theoremalpha}
\vspace{0.2cm}
We sketch the strategy of the proof of Theorem \ref{picxg}. For any p.p.a.v. $A$, we have the following exact sequence of abstract groups
\begin{equation}\label{intro}
0\longrightarrow \text{Hom}_{\ms A_{g,n}}(\ms A_{g,n},\ms X^{\vee}_{g,n})\longrightarrow \bigslant{\Pic(\ms X_{g,n})}{\Pic(\ms A_{g,n})}\stackrel{res}{\longrightarrow} \text{NS}(A)
\end{equation}
where $\ms X_{g,n}^\vee$ is the universal dual abelian variety and the second map is obtained by composing the restriction on the Picard group of the geometric fiber of $\ms X_{g,n}\to\ms A_{g,n}$ corresponding to $(A,\lambda,\varphi)$ with the first Chern class map. 

Using the universal principal polarization, we can identify the kernel with the set of sections of the universal abelian variety and we will prove that it is isomorphic to the group $\left(\bb Z/n\bb Z\right)^{2g}$ of the $n$-torsion points. This was obtained by Shioda in the elliptic case and then in higher dimension by Silverberg when $A_{g,n}$ is a variety, i.e. when $n\geq 3$. We will extend their results to the remaining cases. Then, using the universal principal polarization, we will identify $\left(\bb Z/n\bb Z\right)^{2g}$ with the group of rigidified $n$-roots line bundles.

Then, we will focus on the cokernel. We will fix a Jacobian variety $J(C)$ with Neron-Severi group generated by its theta divisor $\theta$. Since the rigidified canonical line bundle $\mt L_{\Lambda}$ restricted to the Jacobian is algebraically equivalent to $2\theta$, the index of the image of $res$ in $\text{NS}(J(C))$ can be at most two. Then, by studying the existence of a line bundle on $\ms X_{g,n}$ inducing the universal principal polarization, we will show that the inclusion $\text{Im}(res)\subset \text{NS}(J(C))$ is an equality if and only if $n$ is even, concluding the proof of Theorem \ref{picxg}.\\

The paper is organized in the following way. In Section \ref{1}, we recall some known facts about abelian varieties and their moduli spaces. In Section \ref{0-deg}, we give an explicit description of the set of sections of the universal abelian variety $\ms X_{g,n}\to\ms A_{g,n}$. Then in Section \ref{prel}, we prove the exactness of the sequence (\ref{intro}) and we give a proof of Theorem \ref{corxg}. Finally, in Section \ref{thediv} we show that the universal principal polarization of $\ms X_{g,n}\to\ms A_{g,n}$ is induced by a line bundle if and only if $n$ is even.\\

We will work with the category of schemes locally of finite type over the complex numbers. The choice of the complex numbers is due to the fact that our computation is based upon the Shioda-Silverberg's computation of the Mordell-Weil group of $\ms X_{g,n}\to\ms A_{g,n}$ for $n\geq 3$ and the Putman's computation of the Picard group of $\ms A_{g,n}$, which are proved over the complex numbers. Moreover, Shioda proved that, in positive characteristic, the Mordell-Weil group of $\ms X_{1,4}\to\ms A_{1,4}$ can have positive rank (see \cite{shioda2}). So, it seems that our statements are not true in positive characteristic, but we do not have any evidence of this.\vspace{0.3cm}\\
\textbf{Acknowledgements}. This paper originated at the 2015 P.R.A.G.MAT.I.C. workshop, in Catania. We would like to wholeheartedly thank the workshop organizers A. Ragusa, F. Russo and G. Zappal\`a, the lecturers and the collaborators for creating such a nice and productive environment. We are very grateful to Filippo Viviani for suggesting the problem, for many helpful suggestions and for his careful review of the paper. We wish to thank Giulio Codogni for many helpful discussions.

\section{The universal abelian variety $\ms X_{g,n}\to\ms A_{g,n}$.}\label{1}
In this section we will introduce our main object of study: the \emph{universal abelian variety $\ms X_{g,n}$ over the moduli stack $\ms A_{g,n}$ of principally polarized abelian varieties with level n-structure}. Before giving a definition, we need to recall some known facts about the abelian schemes. For more details the reader can refer to  \cite{Mum70} and \cite[Chap. 6, 7]{Mum94}.

\begin{Def} A group scheme $\pi:A\to S$ is called an \emph{abelian scheme} if $\pi$ is smooth, proper and the geometric fibers are connected.
\end{Def}
It is known that an abelian scheme is a commutative group scheme and its group structure is uniquely determined by the choice of the zero section. An homomorphism of abelian schemes is a morphism of schemes which sends the zero section in the zero section.

Let $A\to S$ be a projective abelian scheme of relative dimension $g$ and $O_A$ its zero section. Consider the relative Picard functor
$$
\begin{array}{rcl}
\Pic_{A/S}: (Sch/S)&\longrightarrow& (Grp)\\
T\to S &\mapsto & \bigslant{\Pic(T\times_{S}A)}{\Pic(T)}.
\end{array}
$$
We set $\Pic_{A/S\,(Zar)}$, resp. $\Pic_{A/S\,(Et)}$, resp. $\Pic_{A/S\,(fppf)}$ the associated sheaves with respect to the Zariski, resp. \'Etale, resp. fppf topology. Since $A\to S$ has sections, namely the zero section, and the structure homomorphism is universally an isomorphism, the relative Picard functor and the associated sheaves above are all isomorphic (see \cite[Theorem 9.2.5]{FAG}). Moreover, $\Pic_{A/S}$ is isomorphic to the functor of rigidified (i.e. trivial along the zero section) line bundles on $A\to S$
$$
\begin{array}{rcl}
\Pic'_{A/S}: (Sch/S)&\longrightarrow& (Grp)\\
T\to S &\mapsto & \{ \mt L\in \Pic(T\times_{S}A)| \; O_{T\times_{S}A}^*\mt L\cong\oo_T\}
\end{array}
$$
where $O_{T\times_{S}A}$ is the zero section of $T\times_{S}A\to T$ induced by $O_A$. This functor is represented by a locally noetherian group $S$-scheme $Pic_{A/S}$ (see \cite[Theorem 9.4.18.1]{FAG}), called the \emph{relative Picard scheme}. There is a subsheaf $\op{Pic}^0_{A/S}\subset Pic_{A/S}$ parametrizing rigidified line bundles which are algebraically equivalent to $0$ on all geometric fibers. It is represented by an abelian scheme: the \emph{dual abelian scheme} $A^{\vee}\to S$ \cite[9.5.24]{FAG}. By \cite[9.6.22]{FAG}, the definition of dual abelian scheme in \cite{Mum94} coincides with the definition above. From the theory of the Picard functor of an abelian scheme, we have an homomorphism of group schemes over $S$
\begin{equation}\label{ex}
\begin{array}{rcl}
\lambda:Pic_{A/S}&\to&  Hom_S (A,A^{\vee})\\
\mt L&\mapsto& \left(a\mapsto\lambda(\mt L)(a):=t_a^*\mt L\otimes\mt L^{-1}\right)
\end{array}
\end{equation}
where $t_a:A\to A$ is the translation by $a$ (see \cite[Ch. 6, \S2]{Mum94}). The kernel is the dual abelian scheme $A^{\vee}\to S$. In particular, when $S=\op{Spec}(k)$, with $k$ an algebraically closed field, we can identify the image of $\lambda$ with the Neron-Severi group $\op{NS}(A)$ of the abelian variety $A$. 

\begin{Def}
A \emph{principal polarization} $\lambda$ of a projective abelian scheme $A\to S$ is an $S$-isomorphism $\lambda:A\to A^{\vee}$ such that over the geometric points $s\in S$, it is induced by an ample line bundle on $A_s$ via the homomorphism (\ref{ex}) above.  A \emph{principally polarized abelian scheme $(A\to S,\lambda)$} is a projective abelian scheme $A\to S$ together with a principal polarization $\lambda$.
\end{Def}
We denote with $A[n]$ (resp. $A^\vee[n]$) the group of $n$-torsion points (resp. of $n$-roots of the trivial line bundle $\oo_A$). Up to \'etale base change $S'\to S$, they are isomorphic to the locally constant group scheme $(\bb Z/n\bb Z)^{2g}_{S'}$. Observe that a principal polarization induces an isomorphism of group schemes $A[n]\cong A^{\vee}[n]$. For any principally polarized abelian scheme $(A\to S,\lambda)$, we denote with
$$\textbf{e}_n:A[n]\times A^{\vee}[n]\to\bb{\mu}_{n,S}\quad(\text{resp. }\textbf{e}_n^{\lambda}:A[n]\times A[n]\to \mu_{n,S})$$ the Weil pairing (resp. the pairing obtained by composing the Weil pairing with the isomorphism $(Id_A\times_S \lambda)|_{A[n]}$). The first one is non-degenerate, while the second one is non-degenerate and skew-symmetric. 

For the rest of the paper we fix $\zeta_n$ a primitive $n$-root of the unity over the complex numbers. By this choice we have a canonical isomorphism between the constant group $\bb Z/n\bb Z$ and the group $\mu_{n,\bb C}$ of the $n$-roots of unity, which sends $1$ to $\zeta_n$.

\begin{Def} A \emph{symplectic principal level $n$-structure} (or \emph{level $n$-structure} in short) over a principally polarized abelian scheme $(A\to S,\lambda)$ of relative dimension $g$ is a isomorphism $\varphi:A[n]\cong(\bb Z/n\bb Z)_S^{2g}$ such that $\textbf{e}_n^{\lambda}(a,b)=e(\varphi(a),\varphi(b))$, where $\text{e}:(\bb Z/n\bb Z)^{2g}_S\times(\bb Z/n\bb Z)^{2g}_S\to \mu_{n,S}$ is the standard non-degenerate alternating form on $(\bb Z/n\bb Z)_S^{2g}$ defined by the $2g\times 2g$ square matrix $\bigl(\begin{smallmatrix}
0&I_g \\ -I_g&0
\end{smallmatrix} \bigr)$ composed with the isomorphism $\bb Z/n\bb Z\cong\mu_n$ defined by $\zeta_n$.
\end{Def}

After these definitions we can finally introduce our main objects of study.

\begin{Def}Let $\ms A_{g,n}\to (Sch/\bb C)$ the moduli stack whose objects over a scheme $S$ are the triples $(A\to S,\lambda, \varphi)$ where $(A\to S,\lambda)$ is a principally polarized abelian scheme of relative dimension $g$ with a level $n$-structure $\varphi$. A morphism between two triples $(f,h):(A\to S,\lambda, \varphi)\longrightarrow(A'\to S',\lambda', \varphi')$
is a cartesian diagram
$$
\xymatrix{
A\ar[d]\ar[r]^{f} &A'\ar[d]\\
S\ar[r]^h &S'
}
$$
such that $f(O_A)=O_A'$, where $O_A$ (resp. $O_A'$) is the zero section of $A$ (resp. $A'$), $f^{\vee}\circ\lambda'\circ f=\lambda$ and $\varphi'\circ f_{A[n]}=\left(\text{Id}_{(\bb Z/n\bb Z)^{2g}}\times h\right)\circ\varphi$.
\end{Def}

A proof of the next theorem can be obtained adapting the arguments in \cite[ch. 7]{Mum94}.

\begin{Teo}For any $n\geq 1$, $\ms A_{g,n}$ is an irreducible smooth Deligne-Mumford stack of dimension $\frac{g(g+1)}{2}$. Moreover, if $n\geq 3$ it is a smooth quasi-projective variety.
\end{Teo}

This stack comes equipped with the following objects:
\begin{itemize}
\item[-]  two stacks $\ms X_{g,n}$ and $\ms X_{g,n}^{\vee}$ together with representable proper and smooth morphisms of stacks $\pi:\ms X_{g,n}\to\ms A_{g,n}$ and $\pi^{\vee}:\ms X^{\vee}_{g,n}\to\ms A_{g,n}$ of relative dimension $g$;
\item[-] a closed substack $\ms X_{g,n}[n]\subset\ms X_{g,n}$, which is finite and \'etale over $\ms A_{g,n}$;
\item[-] an isomorphism of stacks $\Lambda:\ms X_{g,n}\to\ms X_{g,n}^{\vee}$ over $\ms A_{g,n}$;
\item[-] an isomorphism of stacks $\varPhi:\ms X_{g,n}[n]\cong (\bb Z/n\bb Z)^{2g}_{\ms A_{g,n}}$ over $\ms A_{g,n}$,
\end{itemize}
such that for any morphism $p:S\to\ms A_{g,n}$ associated to an object $(A\to S,\lambda, \varphi)$, we have two commutative polygons
$$
\xymatrix{
& A\ar[rr]\ar[ld]\ar@{-}[d]^(.70){\lambda}  &&\ms X_{g,n}\ar[ld]^\pi\ar[dd]^(0.4){\Lambda}\\
S\ar[rr]_(0.6){p} &\ar[d] & \ms A_{g,n}&\\
& A^{\vee}\ar[lu]\ar[rr] &&\ms X_{g,n}^{\vee}\ar[lu]^{\pi^{\vee}}
}\quad
\xymatrix{
& A[n]\ar[rr]\ar[ld]\ar@{-}[d]^(.70){\varphi}  &&\ms X_{g,n}[n]\ar[ld]^\pi\ar[dd]^(0.4){\varPhi}\\
S\ar[rr]_(0.6){p} &\ar[d] & \ms A_{g,n}&\\
& (\bb Z/n\bb Z)_{\ms A_{g,n}}^{2g}\ar[lu]\ar[rr] &&\ms (\bb Z/n\bb Z)_{\ms A_{g,n}}^{2g}\ar[lu]
}
$$
where the faces with four edges are cartesian diagrams. In other words $(\ms X_{g,n}\to\ms A_{g,n},\Lambda,\varPhi)$ is the universal triple of the moduli stack $\ms A_{g,n}$. We will call $\ms X_{g,n}\to\ms A_{g,n}$ the \emph{universal abelian variety over $A_{g,n}$} and with $O_{\ms X_{g,n}}$ we will denote its zero section. The isomorphism $\Lambda$ will be called \emph{universal polarization} and $\varPhi$ the \emph{universal level n-structure}. Observe that by definition the stack $\pi^{\vee}:\ms X^{\vee}_{g,n}\to\ms A_{g,n}$ parametrises the line bundles on the universal abelian variety which are algebraically trivial on each geometric fiber. We will call it \emph{universal dual abelian variety}.

\section{The Mordell-Weil group of $\ms X_{g,n}$.}\label{0-deg}
A first step to prove the Theorem \ref{picxg} is to understand the sections of the universal dual abelian variety $\ms X^{\vee}_{g,n}\to\ms A_{g,n}$ restricted to all open substacks $\ms{U} \subseteq \ms{A}_{g,n}$. Using the universal polarization $\Lambda$, this amounts to understanding the group of the rational sections of the universal abelian variety, which is usually called \emph{Mordell-Weil group}. We want to prove the following:

\begin{Teo}\label{strf0}
Assume $g\geq 1$. For all open substacks $\ms{U} \subseteq \ms{A}_{g,n}$ the group of sections $\ms{U}\rightarrow \ms{X}_{g,n}\times_{\ms{A}_{g,n}}\ms{U}$ is isomorphic to $(\mathbb{Z}/n\mathbb{Z})^{2g}$, the isomorphism being given by restricting the canonical isomorphism  $\varPhi:\ms X_{g,n}[n]\cong (\bb Z/n\bb Z)^{2g}_{\ms A_{g,n}}$.
\end{Teo}

In this section we will assume implicitly $g\geq 1$. For $n\geq 3$ the Theorem was proven by Shioda in the elliptic case and then by Silverberg in higher dimension (see Theorem \ref{Silverberg} below). Starting from this, we will extend the Shioda-Silverberg's results to the remaining cases. The main problem here for $n=\lbrace 1,2 \rbrace$ is that differently from the case of $\ms{A}_{g,m} m\geq 3$, the stack $\ms{A}_{g,n}$ is not generically a scheme, so we cannot really reduce the argument to considerations on the fiber of the generic point, or more precisely, there is not generic point at all. To solve this we introduce the technique of equivariant approximation:

\begin{Prop-Def}\label{Equi}
Let $G$ be an affine smooth group scheme, and let $\ms{M}=\left[ X/G \right]$ be a quotient stack. Choose a representation $V$ of $G$ such that $G$ acts freely on an open subset $U$ of $V$ whose complement has codimension $2$ or more. The quotient $\left[ X\times U/G \right]$ will be called an \emph{equivariant approximation of $\ms M$}. It has the following properties:
\begin{enumerate}
\item $\left[ X\times U/G \right]$ is an algebraic space. If $X$ is quasiprojective and the action of $G$ is linearized then $\left[ X\times U/G \right]$ is a scheme.
\item $\left[X\times V/G \right]$ is a vector bundle over $\left[X/G\right]$, and $\left[ X\times U /G \right]\hookrightarrow \left[X\times V/G \right]$ is an open immersion, whose complement has codimension $2$ or more.
\item The map $\left[ X\times U/G\right] \rightarrow \left[X/G \right]$ is smooth, surjective, separated and if $K$ is an infinite field then every map $\op{Spec}(K) \xrightarrow{p} \left[ X/G \right]$ lifts to a map $\op{Spec}(K)\rightarrow [X\times U]/G$.
\item If $X$ is locally factorial, then the map $\left[X\times U/G\right]\to \left[X/G\right]$ induces an isomorphism at level of Picard groups.
\end{enumerate}

Moreover, such a representation $V$ always exists for an affine smooth groups scheme $G$ over a field $k$.
\end{Prop-Def}
\begin{proof}
This is presented in \cite{Equiv}, where all points above are proven. The only point that needs further commenting is point $3$. Let $P: \op{Spec}(K) \rightarrow \left[ X/G \right]$, and consider the fiber $\left[ X \times U / G \right] \times_{\left[ X/G \right]} \op{Spec}(K)$. It is an open subset of a vector bundle over $\op{Spec}(K)$, so if $K$ is infinite we know that its rational points are dense, and we have infinitely many liftings of $P$.

 Note if $K$ is a finite field it is possible for the fiber $\left[ X \times U /G \right]\times_{\left[ X/G \right]} \op{Spec}(K)$ to have no rational point at all, as the rational points in a vector bundle over a finite field form a closed subset. This shows that if $K$ is finite the point may not have any lifting.
 
The last statement is a direct consequence of the well-known fact that an affine smooth algebraic group over a field $k$ always admits a faithful finite dimensional representation, so we just need to prove it for $G=GL_n$. Such a representation can be constructed for example as in \cite[3.1]{Equiv}.
\end{proof}

Note that there exists $r$ such that $\ms{A}_{g,n}$ is the quotient a quasiprojective scheme by a linearized group action of the affine group scheme $PGL_r$ \cite[ch. 7]{Mum94}. Then we can take an equivariant approximation $B_{g,n} \xrightarrow{\pi} \ms{A}_{g,n}$ where $B_{g,n}$ is a scheme, and by pulling back $\ms X_{g,n}$ we get an induced family $X_{g,n} \rightarrow B_{g,n}$.

\begin{Prop}
If $\ms{X}_{g,n} \rightarrow \ms{A}_{g,n}$ has two non isomorphic sections over some open subset $\ms U \subseteq \ms{A}_{g,n}$, then $X_{g,n} \rightarrow B_{g,n}$ has two non isomorphic (i.e. distinct, as $X_{g,n}, B_{g,n}$ are schemes) sections over the open subset $U := \ms U \times_{\ms A_{g,n}} B_{g,n}$.
\end{Prop}
\begin{proof}
Let $\ms{U}, U$ be as above. Let $\sigma_1, \sigma_2$ be two non isomorphic sections of $\ms{X}_{g,n} \rightarrow \ms{A}_{g,n}$. By the universal property of fibered product we get induced sections $(\op{Id},\sigma_1\circ \pi),(\op{Id}, \sigma_2\circ \pi):B_{g,n} \rightarrow X_{g,n}$. We get an obvious commutative diagram, and we can conclude that the two maps must be different since $\pi$ is an epimorphism, being a smooth covering.
\end{proof}

This way we have reduced our problem to showing that there are exactly $n^{2g}$ sections of $X_{g,n} \rightarrow B_{g,n}$. We now proceed to prove some lemmas.

\begin{Lem}
Let $S$ be the spectrum of a DVR $R$, and let $A \rightarrow S$ be an abelian scheme. Let $p,P$ be respectively the closed and generic point of $S$. Then for all $m$ the order of the $m$-torsion in the Mordell-Weil group of the closed fiber $A_p$ is greater or equal than the order of the $m$-torsion in the Mordell Weil group of the generic fiber $A_P$. 
\end{Lem}
\begin{proof}
Let $A\left[ m\right] \rightarrow S$ be the closed subscheme of $n$-torsion points of $A$. Then $A\left[ m\right] \xrightarrow{\pi} S$ is a proper \'etale morphism, as we are working in characteristic zero. We may suppose that $R$ is Henselian. Being \'etale and proper, the map $\pi$ is finite. A finite extension of a local Henselian ring is a product of local Henselian rings \cite[Tag 04GH]{StPr}. Then for every lifting of $P$ to $A\left[ m \right]$ we have a corresponding map of local rings $R' \rightarrow R$ that is \'etale of degree one, i.e. it is an isomorphism. This implies that there is a corresponding lifting of $p$ to $A\left[ m\right]$, proving the lemma.
\end{proof}

\begin{Lem}\label{local}
Let $R$ a Noetherian local regular ring, and let $p,P$ be the closed and generic and closed points of $\op{Spec}(R)$. Let $A \rightarrow \op{Spec}(R)$ be an abelian scheme. Then for all $m > 0$ the order of the $m$-torsion in the Mordell-Weil group of $A_p$ is greater or equal than the order of the $m$-torsion in the Mordell Weil group of $A_P$. 
\end{Lem}
\begin{proof}
We prove the lemma by induction on the dimension of $\op{Spec}(R)$. The case $\op{dim}(\op{Spec}(R))=1$ is the previous lemma. Now suppose $\op{dim}(S) \geq 2$ and take a regular sequence $(a_1,\ldots,a_r)$ for $R$. Then $R_1=R_{(a_1)}$ and $R_2:=R/(a_1)$ are both Noetherian local regular rings. If we see $\op{Spec}(R_1),\op{Spec}(R_2)$ as subschemes of $\op{Spec}(R)$ the generic point of $\op{Spec}(R_1)$ is $P$, the generic point of $\op{Spec}(R_2)$ is the closed point of $\op{Spec}(R_1)$, and the closed point of $\op{Spec}(R_2)$ is $p$.

 Denote respectively by $P_1,P_2$ the generic points of $\op{Spec}(R_1),\op{Spec}(R_2)$, and by $p_1,p_2$ their closed points. Denote respectively by $A_1, A_2$ the pullbacks of $A$ to $\op{Spec}(R_1), \op{Spec}(R_2)$. Then
$$ \sharp (A\left[ m\right](P))=\sharp(A_1\left[ m\right](P_1))\leq \sharp(A_1\left[ m\right](p_1))= \sharp(A_2\left[ m\right](P_2)) \leq \sharp(A_2\left[ m\right](p_2))=\sharp(A\left[ m\right](p)) $$
where the first equality comes from previous lemma, and the last comes from the inductive hypothesis.
\end{proof}

\begin{Lem}
For all $n>0, g>0$ there is a field $K$, finitely generated over $\bb C$, and a p.p.a.v. of dimension $g$ with level $n$-structure over $K$ that has Mordell-Weil group isomorphic to $(\mathbb{Z}/n\mathbb{Z})^{2g}$, where the isomorphism comes from the level structure.
\end{Lem}
\begin{proof}
For $g=1$ it the result is proven in \cite[Prop. 3.2]{Schroer} when $N=1$ and in \cite[Thm 5.5 + Rmk 5.6]{Shioda1} for the other cases (see also \cite{shioda2}). We can then take powers of these elliptic curves to obtain the general statement.
\end{proof}

Recall now that $\ms X_{g,n} \rightarrow \ms A_{g,n}$ is a smooth morphism with connected fibers, and $\ms A_{g,n}$ is smooth and irreducible, so it the same goes for $\ms X_{g,n}$. Consequently also $B_{g,n}$ and $X_{g,n}$ are smooth and irreducible, being open subsets of vector bundles over $\ms A_{g,n}$ and $\ms X_{g,n}$ respectively.

\begin{Prop}
Let $\xi$ be the generic point of $B_{g,n}$. Then the torsion of the Mordell-Weil group of $X_{g,n} \times_{B_{g,n}} \xi$ is isomorphic to $(\mathbb{Z}/n\mathbb{Z})^{2g}$, the morphism coming from the level structure.
\end{Prop}
\begin{proof}
Consider any point $P$ in $B_{g,n}$ such that $X_{g,n} \times_{B_{g,n}} P$ Mordell-Weil group with torsion exactly $(\mathbb{Z}/n\mathbb{Z})^{2g}$. This exists by the previous lemma and the fact that the map $B_{g,n} \rightarrow \ms{A}_{g,n}$ has the lifting property for points (\ref{Equi}, point 4). Then we can apply Lemma (\ref{local}) to the local ring $\mathcal{O}_{B_{g,n},P}$ and the fact that $X_{g,n} \times_{B_{g,n}} \xi$ comes with a canonical isomorphism of $(\bb Z / n \bb Z)^{2g}$ with its $n$-torsion to conclude. 
\end{proof}

Half of our work towards theorem (\ref{strf0}) is done. Now we need to show that the Mordell Weil group of the generic fiber $X_{g,n} \times_{B_{g,n}} \xi$ is torsion, so that it will be equal to $(\mathbb{Z}/n\mathbb{Z})^{2g}$. Next theorem gives us an answer when $\ms A_{g,n}$ is a variety.

\begin{Teo}[Shioda-Silverberg]\label{Silverberg}
Suppose $n \geq 3$. Let $\xi \rightarrow \ms{A}_{n,g}$ be the generic point. Then the Mordell-Weil group of $\ms{X}_{n,g}\times_{\ms{A}_{n,g}} \xi$ is isomorphic to $(\bb Z/n\bb Z)^{2g}$, the isomorphism coming from the level structure.
\end{Teo}
\begin{proof} For a complete proof in the elliptic case (resp. higher dimension) we refer to \cite{Shioda1} (resp. \cite{Sil85}). The statement with a sketch of the proof can be found in \cite[Theorem 1 and 3, pp. 227-235]{Lox90}.
\end{proof}

For $n=1,2$ we need a few more steps.

\begin{Lem}\label{transc}
Let $A_F$ be a principally polarized abelian variety over a field $F$ and let $Q$ be a finitely generated purely transcendental extension of $F$. Define $A_Q := A_F \times_F Q$. Then the homomorphism of Mordell-Weil groups $A_F(F) \rightarrow A_Q(Q)$ is an isomorphism.
\end{Lem}
\begin{proof}
Since $A_F$ is principally polarized, the Mordell-Weil group of $A$ is isomorphic to $\Pic^0_{A_F/F}(F)$, and the Mordell-Weil group of $A_Q$ is isomorphic to $\Pic^0_{A_Q}(Q)$. Both $A_Q$ and $A_F$ have a rational point, so we have $\Pic_{A_F/F}(F)=\Pic(A_F)$, $\Pic_{A_Q/Q}(Q)=\Pic(A_Q)$ (see for example \cite[Remark 9.2.11]{FAG}). Moreover $A_Q$ and $A_F$ are smooth and thus locally factorial, so their Picard groups are isomorphic to the group of divisors modulo rational equivalence. Consider the following commutative triangle:

$$\xymatrixcolsep{5pc}
\xymatrix{& A_F \times \bb A^n_F \ar@{->}[d]^{\pi} \ar@{<-}[dl]_{i}\\ 
A_Q \ar@{->}[r] & A_F}$$

Here $n$ is the degree of transcendence of $Q/F$ and the map $i$ is the inclusion of the generic fiber. The pullback through $\pi$ is an isomorphism on Picard groups. The pullback through $i$ is surjective. The map $Q \rightarrow F$ is a smooth covering, so the pullback $\Pic(A_F)=\Pic_{A_F /F}(F) \rightarrow \Pic_{A_Q/Q}(Q)=\Pic(A_Q)$ is injective. This shows that the pullback through $A_Q \rightarrow A_F$ induces an isomorphism on Picard groups.

Now recall that $\Pic_{A_Q/Q}$ is isomorphic to $\Pic_{A_F/F} \times_F Q$ as a $Q$-scheme, and in particular the map $\Pic_{A_Q/Q}\rightarrow \Pic_{A_F/F}$ has connected fibers. This implies that if $L \in \Pic(A_F)$ pulls back to $L' \in \Pic^0_{A_Q}$ then $L$ must belong to $\Pic^0_{A_F}$, and the isomorphism on the Picard groups then implies the isomorphism on the $\Pic^0$.

We can now conclude as by definition the image of a point $p \in \Pic_{A_Q/Q}(Q)$ representing the pullback of a line bundle $L \in \Pic_{A_F/F}(F)$ is $L$ itself. Then putting everything together we proved that the pull-back along the map $A_Q\to A_F$ induces an isomorphism on the Mordell-Weil groups, which proves our claim.
\end{proof}

\begin{Prop}\label{wkfrequiv}
The generic fiber of $X_{g,n} \rightarrow B_{g,n}$ has Mordell-Weil group equal to $(\mathbb{Z}/n\mathbb{Z})^{2g}$.
\end{Prop}
\begin{proof}
Consider the following cartesian cube:

\begin{center}

\begin{tikzcd}[back line/.style={densely dotted}, row sep=2em, column sep=2em]
& X_{g,3n} \ar{dl}[swap]{\phi'_n} \ar{rr}{\rho'_n} \ar{dd}[near end]{\pi'_n} 
  & & B_{g,3n} \ar{dd}{\pi_n} \ar{dl}[swap,sloped,near start]{\phi'} \\
X_{g,n} \ar[crossing over]{rr}[near start]{\rho'} \ar{dd}[swap]{\pi'} 
  & & B_{g,n} \\
& \ms{X}_{g,3n} \ar[near start]{rr}{\rho_n} \ar{dl}{\phi'} 
  & & \ms{A}_{g,3n} \ar{dl}{\phi} \\
\ms{X}_{g,n} \ar{rr}{\rho} & & \ms{A}_{g,n} \ar[crossing over, leftarrow, near start]{uu}{\pi}
\end{tikzcd}

\end{center}

The $\pi$ maps are equivariant approximations, the $\phi$ maps are \'etale finite, the $\rho$ maps are families of abelian varieties. Let $\zeta$ be the generic point of $B_{g, 3n}$. First we want to understand the Mordell-Weil group of the generic fiber $X_{g,3n} \times_{B_{g,3n}} \zeta$. As $B_{g,3n}$ is an open subset of a vector bundle over $\ms{A}_{g,3n}$ the generic point of $B_{g,3n}$ is a purely transcendental extension of the generic point of $\ms{A}_{g,3n}$. Then by Lemma (\ref{transc}) we can conclude that the Mordell-Weil group of $X_{g,3n} \times_{B_{g,3n}} \zeta$ is isomorphic to the Mordell-Weil group of the generic fiber of $\ms{X}_{g,3n} \rightarrow \ms{A}_{g,3n}$. As $3n$ is greater or equal to three he latter is torsion due to Silverberg's theorem.

Now we already know that the Mordell-Weil group of $X_{g,n} \times_{B_{g,n}} \xi$ has torsion equal to $(\mathbb{Z}/n\mathbb{Z})^{2g}$, and as the \'etale map $X_{g,n} \times_{B_{g,3n}} \zeta \rightarrow X_{g,n} \times_{B_{g,n}} \xi$ is an epimorphism it must also inject into the Mordell-Weil group of $X_{g,3n} \times_{B_{g,3n}} \zeta$. The latter is torsion, so the Mordell-Weil group of $X_{g,n} \times_{B_{g,n}} \xi$ must be equal to $(\mathbb{Z}/n\mathbb{Z})^{2g}$.
\end{proof}

\begin{proof}[Proof of theorem \ref{strf0}] any two sections $\ms{U} \times_{\ms{A}_{g,n}} B_{g,n} \rightarrow \ms{U}\times_{\ms{A}_{g,n}} X_{g,n}$ that induce the same rational point in the generic abelian variety above must be generically equal. But two maps from an irreducible and reduced scheme to a separated scheme that are generically equal must be the same \cite[Tag 0A1Y]{StPr}. This, in addition to the fact that by definition there exist $n^{2g}$ canonical distinct sections of $\ms{X}_{g,n}\rightarrow \ms{A}_{g,n}$, coming from the level structure, concludes the proof of our theorem.
\end{proof}

\section{Preliminaries on the Picard group of $\ms X_{g,n}$.}\label{prel}

Let $g\geq 2$ and $n\geq 1$. In this section, we will give a partial description of the group of the rigidified line bundles on $\ms X_{g,n}$, which allows us to prove Theorem \ref{corxg}.\vspace{0.1cm}

First of all, we introduce some natural line bundles on the universal abelian variety.

\begin{Def}\label{rigi}
We will call \emph{rigidified $n$-roots line bundles} of $\ms X_{g,n}\to\ms A_{g,n}$ the rigidified line bundles over $\ms X_{g,n}$  which are $n$-roots of the trivial line bundle $\oo_{\ms X_{g,n}}$.
\end{Def}

Observe that the rigidified $n$-roots line bundles correspond to the sections of a substack of $\ms X^{\vee}_{g,n}$, which is finite and \'etale over $\ms A_{g,n}$. Over any $\bb C$-point $(A,\lambda,\varphi)$, they correspond to the $n$-torsion points $A^{\vee}[n]$ of the dual abelian variety. Using the universal polarization $\Lambda:\ms X_{g,n}\cong\ms X^{\vee}_{g,n}$, we obtain an isomorphism between the rigidified $n$-roots line bundles and the the group of $n$-torsion sections of $\ms X_{g,n}\to\ms A_{g,n}$. Using the universal level $n$-structure $\varPhi$, we see immediately the this group is isomorphic to $\left(\bb Z/n\bb Z\right)^{2g}$.

\begin{Def}
Let $(A\to S,\lambda)$ be a family of p.p.a.v. over $S$. Let $A^{\vee}$ the dual abelian scheme and let $\mt P$ the rigidified Poincar\'e line bundle on $A\times_S A^{\vee}$. Pulling back $\mt P$ through the map $(Id_A,\lambda):A\rightarrow A\times_S A^{\vee}$, we get a rigidified line bundle on $A$. Since this line bundle is functorial in $S$, it defines a line bundle over the universal abelian variety $\ms X_{g,n}$: we will call this sheaf \emph{rigidified canonical line bundle $\mt L_{\Lambda}$}.
\end{Def}
\begin{Oss}\label{oss} 
Let $(A,\lambda,\varphi)$ be a $\mathbb C$-point in $\ms A_{g,n}$. There exists a line bundle $\mt M$, unique up to translation, over $A$ inducing the polarization. The line bundle $\mt L_{\Lambda}$, restricted to a $(A,\lambda,\varphi)$ is equal to $\mt M^2$ in $\op{NS}(A)$.

Indeed, this is equivalent to showing that $\lambda(\mt L_{\Lambda}|_{(A,\lambda,\varphi)}\otimes \mt M^{-2})=0$. By \cite[Proposition 6.1]{Mum94}, we have that $\lambda(\mt L_{\Lambda})=2\Lambda$. Then, by definition of universal polarization, $\Lambda|_{(A,\lambda,\varphi)}=\lambda=\lambda(\mt M)$ from which the assertion follows immediately.
\end{Oss}

The proof of next lemma can be found in \cite[Proposition 6.1]{Mum94}.

\begin{Lem}\label{rigidity} Given a commutative diagram of schemes
$$
\xymatrix{
X\ar[dr]_p\ar[rr]^f & & Y\ar[dl]^q\\
& S&}
$$
where $X$ is an abelian scheme over a connected scheme $S$. If, for one point $s\in S$, $f(X_s)$ is set-theoretically a single point, then there is a section $0:S\to Y$ such that $f=0\circ p$.
\end{Lem}

Now we are going to study the image of the homomorphism (\ref{ex}).

\begin{Prop}\label{inj}Let $A$ be an abelian scheme over $S$ and $\mt L$ a line bundle on $A$. Suppose that there exists a closed point $s\in S$ such that $\mt L_s=\oo_{A_s}$ in $NS(A_s)$. Then $\lambda(\mt L)$ is the zero homomorphism in $\op{Hom}_S(A,A^{\vee})$.
\end{Prop}

\begin{proof} By hypothesis $\lambda(\mt L)_s=\lambda(\mt L_s)$ is the zero homomorphism. By Lemma \ref{rigidity}, $\lambda(\mt L)$ factorizes as the structure morphism $A\to S$ and a section of the dual abelian scheme. Since $\Lambda$ is an homomorphism of abelian schemes, it must be the zero homomorphism.
\end{proof}

Let $(A,\lambda,\varphi)$ be a $\mathbb C$-point of $\ms A_{g,n}$. Consider the homomorphism
$$
res:\bigslant{\Pic(\ms X_{g,n})}{\Pic(\ms A_{g,n})}\longrightarrow \Pic(A)\longrightarrow \text{NS}(A)
$$
where the first row is given by restriction and the second one is the first Chern class map. We have the following

\begin{Prop}\label{incl}
For any $\mathbb C$-point $(A,\lambda,\varphi)$ in $\ms A_{g,n}$, we have an exact sequence of abstract groups
\begin{equation}
0\longrightarrow \left(\bigslant{\mathbb Z}{n\mathbb Z}\right)^{2g}\longrightarrow \bigslant{\Pic(\ms X_{g,n})}{\Pic(\ms A_{g,n})}\stackrel{res}{\longrightarrow} \text{NS}(A)
\end{equation}
where the kernel is the group of the rigidified $n$-roots line bundles.
\end{Prop}
\begin{proof}Consider the cartesian diagram
$$
\xymatrix{
X_{g,n}\ar[r]\ar[d]& \ms X_{g,n}\ar[d]\\
B_{g,n}\ar[r] &\ms A_{g,n}
}$$
where $B_{g,n}\to \ms A_{g,n}$ is an equivariant approximation as in the Definition-Proposition \ref{Equi}. It induces a commutative diagram of Picard groups
$$
\xymatrix{
\Pic( X_{g,n})& \Pic(\ms X_{g,n})\ar[l]\\
\Pic(B_{g,n})\ar[u] &\Pic(\ms A_{g,n}).\ar[l]\ar[u]
}
$$
We can easily see that $X_{g,n}\to \ms X_{g,n}$ is an equivariant approximation for $\ms X_{g,n}$. In particular, in the diagram of Picard groups, the horizontal arrows are isomorphisms (by Definition-Proposition \ref{Equi}(4)) and the vertical ones are injective.

Let $s$ be a lifting of $(A,\lambda,\varphi)$ over $B_{g,n}$, which exists by Definition-Proposition \ref{Equi}(3). 
Consider the homomorphism of groups
$$
Pic_{X_{g,n}/B_{g,n}}(B_{g,n})\longrightarrow Pic_{X_{g,n}/B_{g,n}}(s)=\Pic(A)\longrightarrow \text{NS}(A)
$$
where the first row is given by restriction and the second one is the first Chern class map. Proposition \ref{inj} implies that
the sequence of groups
$$
0\longrightarrow X_{g,n}^{\vee}(B_{g,n})\longrightarrow Pic_{X_{g,n}/B_{g,n}}(B_{g,n})\longrightarrow \text{NS}(A)
$$ is exact. As observed in Section \ref{1}, we can identify the abstract group $Pic_{X_{g,n}/B_{g,n}}(B_{g,n})$ with $\Pic(X_{g,n})/\Pic(B_{g,n})$. By the diagram above, it is also isomorphic to the group $\Pic(\ms X_{g,n})/\Pic(\ms A_{g,n})$. The assertions about the kernel follows from the results of Section \ref{0-deg}.
\end{proof}

Using this we can complete the description of the Picard group of the universal abelian variety without level structure.\\\\
\begin{dimo} \emph{of Theorem \ref{corxg}}. By \cite[Lemma p. 359]{ACGH} there exists a Jacobian variety $J(C)$ of a smooth curve of genus $g$ with Neron-Severi group generated by the theta divisor $\theta$. We set $m$ the index of the map $res$ in the Proposition \ref{incl} with $(A,\lambda,\varphi)=(J(C),\theta,\varphi)$. Consider the morphism of complexes
$$
\xymatrix{
0\ar[r]&\Pic(\ms A_g)\ar[r]\ar[d]& \Pic(\ms X_g)\ar[r]\ar[d]& m\cdot NS(J(C))\ar[r]\ar[d]& 0\\
0\ar[r]&\Pic(\ms M_g)\ar[r]& \Pic(\ms J_g)\ar[r]& 2\cdot NS(J(C))\ar[r]& 0
}
$$
The top sequence is exact by Proposition \ref{incl} in the case $n=1$. The exactness of the bottom sequence comes from \cite[Theorem 1]{Kou91} (see also \cite{MV}[Subsection 7]). Observe that the last vertical map is surjective by the existence of the line bundle $\mt L_{\Lambda}$ (see Remark \ref{oss}). It must be also injective, because otherwise we can construct a line bundle on $\mt J_g$ which generates the Neron-Severi of $J(C)$. Then $m=2$ and the first assertion follows immediately. As recalled in the introduction the first vertical is an isomorphism when $g\geq 3$, so the second assertion will follow  by the snake lemma.
\end{dimo}

\section{The universal theta divisor.}\label{thediv}
The main result of this section is the following

\begin{Teo}\label{polex}Assume that $g\geq 4$. There exists a line bundle over the universal abelian variety $\ms X_{g,n}\to\ms A_{g,n}$ inducing the universal polarization if and only if $n$ is even. Moreover, if it exists, it can be chosen symmetric.
\end{Teo}

The above theorem allows us to conclude the description of the Picard group of the universal abelian variety.\\\\
\begin{dimo} \emph{of Theorem \ref{picxg}}. As in the proof of Theorem \ref{corxg}, we fix a Jacobian variety $J(C)$ of a smooth curve of genus $g$ with Neron-Severi group generated by the theta divisor $\theta$. By Proposition \ref{incl} with $(A,\lambda,\varphi)=(J(C),\theta,\varphi)$, it is enough to compute the index of the image of the map
$\Pic(\ms X_{g,n})\to NS(J(C))=\bb Z[\theta]$.
By Remark \ref{oss}, the subgroup generated by $\mt L_{\Lambda}$ has index two in $NS(J(C))$. So the theorem follows from Theorem \ref{polex}.
\end{dimo}\\

The rest of the section is devoted to prove Theorem \ref{polex}. 

\begin{Oss}
The sufficient condition is well-known when $\ms A_{g,n}$ is a variety (i.e. $n\geq 3$): see for example the survey of Grushevsky and Hulek (\cite[Section 1]{GH13}) for a good introduction, following \cite{Igu72}. Due to the ignorance of the authors, it is not clear if the results can be extended to the remaining case $n=2$, using the same arguments of \emph{loc. cit.} For this reason, we give a new proof of this fact, following the arguments of Shepherd-Barron in \cite[\S 3.4]{SB08}. Such proof works also when $2\leq g\leq 4$.

Instead the proof of the necessary condition uses a result of Putman  \cite{Pu2012}, which implies that the Hodge line bundle does not admit roots on the Picard group of $\ms A_{g,n}$ (modulo torsion) when $n$ is odd and $g\geq 4$. Anyway, by the remarks that follow \cite[Theorem E]{Pu2012}, it seems that the same holds also in genus two and three, but we do not have any reference of this. For this reason, in this section, we will assume $g$ greater than three.
\end{Oss}
First we will resume some results and definitions from \cite[\S 3.4]{SB08}.

\begin{Def}An \emph{abelian torsor} $(A\curvearrowright P\to S)$ is a projective scheme $P$ over $S$ which is a torsor under an abelian scheme $A\to S$. An abelian torsor is \emph{symmetric} if the action of $A$ on $P$ extends to an action of the semi-direct product $A\rtimes (\mathbb Z/2\mathbb Z)_S$ where $(\mathbb Z/2\mathbb Z)_S$ acts as the involution $i$ on $A$. We will denote with $Fix_P$ the closed subscheme of $P$ where $i$ acts trivially. Note that $Fix_P$ is a torsor under the subscheme $A[2]\subset A$ of the $2$-torsion points.
\end{Def}

The (fppf) sheaf $Pic^{\tau}_{P/S}$ of line bundles, which are the numerically trivial line bundles on each geometric fiber, is represented by the dual abelian scheme $A^{\vee}$. In particular, an ample line bundle $\mt M$ on an abelian torsor $P\to S$ defines a polarization $\lambda$ on $A$ by sending a point $a\in A$ to $t_a^*\mt M\otimes\mt M^{-1}\in Pic^{\tau}_{P/S}=A^{\vee}$, where $t_a:P\to P$ is the translation by $a$.

\begin{Def}A relative effective divisor $D$ on the abelian torsor $(A\curvearrowright P\to S)$ is \emph{principal} if the line bundle $\oo(D)$ defines a principal polarization on $A$. A \emph{principal symmetric abelian torsor (p.s.a.t.)} is a symmetric abelian torsor with an effective principal divisor that is symmetric, i.e. it is $i$-invariant as hypersurface. A \emph{level $n$-structure} on $(A\curvearrowright P\to S)$ is a level $n$-structure $\varphi:A[n]\cong \left(\mathbb Z/n\mathbb Z\right)^{2g}_S$ on $A$.
\end{Def}

Adapting the Alexeev's idea \cite{Al02} of identifying the stack of p.p.a.v with the stack of torsors with a suitable divisor, Shepherd Barron in \cite{SB08} proves that the stack $\ms A_{g}$ is isomorphic to the stack of principal symmetric abelian torsors (p.s.a.t). The proof works also if we add the extra datum of the level structure, obtaining the following

\begin{Prop}\cite[Proposition 2.4]{SB08}. The stack $\ms A_{g,n}$ is isomorphic to the stack whose objects over a scheme $S$ are the principal symmetric abelian torsors with level $n$-structure. A morphism between two objects $$(f,g,h):(A\curvearrowright P\to S, D, \varphi)\rightarrow (A'\curvearrowright P'\to S', D', \varphi')$$ are two cartesian diagrams
$$\xymatrix{
P\ar[d]\ar[r]^{g} &P'\ar[d]\\
S\ar[r]^h &S'
}\quad
\xymatrix{
A\ar[d]\ar[r]^{f} &A'\ar[d]\\
S\ar[r]^h &S'
}$$
where $(f,g,h)$ is a morphism of abelian torsors such that $\varphi'\circ f|_{A[n]}=\left(\text{Id}_{(\bb Z/n\bb Z)^{2g}}\times h\right)\circ\varphi$ and $g^{-1}(D')=D$.
\end{Prop}

Let $\ms N_{g,n}$ be the stack of the p.p.a.v. with a symmetric divisor defining the polarization. The forgetful functor $\ms N_{g,n}\to \ms A_{g,n}$ is a $\ms X_{g,n}[2]$-torsor. In particular, the universal abelian variety $\ms X_{g,n}\to\ms A_{g,n}$ admits a universal symmetric divisor inducing the universal polarization if and only if the torsor $\ms N_{g,n}\to \ms A_{g,n}$ admits a section. We can identify the stack $\ms N_{g,n}$ with the stack of 4-tuples $$(A\curvearrowright P\to S, D,\varphi,x)$$ where $(A\curvearrowright P\to S, D)$ is p.s.a.t with a level $n$-structure $\varphi$ and $x$ is a section of $Fix_P\to S$. A morphism between two objects is a morphism of $\ms A_{g,n}$ compatible with the section of the $i$-invariant locus. Using this interpretation, if we call $\ms P_{g,n}$ the universal abelian torsor on $\ms A_{g,n}$, our problem is equivalent to showing that the $\ms X_{g,n}[2]$-torsor $Fix_{\ms P_{g,n}}\to \ms A_{g,n}$ has a section.\vspace{0.1cm}

We are now going to give another modular description of $\ms N_{g,n}$ in terms of theta characteristics.

\begin{Def}
Let $(A\curvearrowright P\to S, D)$ be a p.s.a.t. and $\lambda:A\cong A^{\vee}$ the principal polarization induced by $\oo(D)$. Let $\mt T_P$ be the subsheaf of $ Hom_S(A[2],\mu_{2,S})$ of morphisms $c$ such that $$c(a)c(b)c(a+b)=\textbf{e}^{\lambda}_2(a,b)$$ for any $a,b\in A[2]$. Any morphism with this property will be called \emph{theta characteristic} of the p.s.a.t. $(A\curvearrowright P\to S, D)$.
The sheaf $\mt T_P$ is a torsor under the action of $A[2]$: $b.c(a)=e_2^{\lambda}(b,a)c(a)$ for $a,b\in A[2]$ and $t\in\mt T_P$. We will call $\mt T_P$ the \emph{torsor of theta characteristics} of the p.s.a.t $(A\curvearrowright P\to S, D)$.
\end{Def}

We have the following

\begin{Prop}\label{thetacar2}The stack $\ms N_{g,n}$ is isomorphic to the stack $\ms T_{g,n}$ which parametrizes the $4$-tuples $(A\curvearrowright P\to S, D,\varphi, c)$ where $(A\curvearrowright P\to S, D)$ is a p.s.a.t. over $S$ with a level $n$-structure $\varphi$ and theta characteristic $c\in \mt T_P$. A morphism between two objects $$(f,g,h):(A\curvearrowright P\to S, D, \varphi,c)\rightarrow (A'\curvearrowright P'\to S', D', \varphi',c')$$ is a morphism on $\ms A_{g,n}$ such that $c'\circ f|_{A[2]}=\left(\text{Id}_{\mu_{2}}\times h\right)\circ c$.
\end{Prop}

First of all, we recall some preliminaries facts. Let $A\to S$ be an abelian scheme with a symmetric line bundle. There exists a unique isomorphism $\varphi:\mt L\cong i^*\mt L$ such that $O_A^*\varphi$ is the identity. 
For any $x\in A[2](S)$ the isomorphism $x^*\varphi:x^*\mt L\cong x^*i^*\mt L=x^*\mt L$ is a multiplication by an element $e^{\mt L}(x)$ of $H^0(S,\oo^*_S)$ and it satisfies the following properties
\begin{enumerate}[(i)]
\item $e^{\mt L}(O_A)=1_S$ and $e^{\mt L}(x)\in\mu_2(S)$ for any $x\in A[2](S)$.
\item $e^{\mt L\otimes \mt M}(x)=e^{\mt L}(x)\cdot e^{\mt M}(x)$ for any $x\in A[2](S)$ and for any symmetric line bundle $\mt M$.
\item $e^{t^*_y\mt L}(x)=e^{\mt L}(x+y)\cdot e^{\mt L}(y)$ where $t_y:A\to A$ is the traslation by $y\in A[2](S)$.
\item If $\mt L$ is such that $\mt L^2\cong \oo_A$ then $e^{\mt L}(x)=\textbf{e}_2(x,\mt L).$
\end{enumerate}
These properties (and their proofs) are slight generalizations of the ones in \cite[pp. 304-305]{Mum66}.\\\\
\begin{dimo}\emph{Proposition \ref{thetacar2}}. To prove the proposition it is enough to show that for any p.s.a.t. $(A\curvearrowright P\to S, D)$ there exists a canonical isomorphism $\phi$ of $A[2]$-torsors between $Fix_P$ and the torsor of theta characteristics $\mt T_P$, such that, for any morphism $
(f,g,h):(A\curvearrowright P\to S, D)\rightarrow (A'\curvearrowright P'\to S', D')
$ in $\ms A_{g}$, we have that $ \phi'(g(\delta))\circ f|_{A[2]}=\left(Id_{\mu_2}\times h\right)\circ\phi(\delta)$ for any $\delta\in Fix_P$.\vspace{0.2cm}

Let $(A\curvearrowright P\to S,D)$ be a p.s.a.t. Let $T$ an $S$-scheme and $\delta\in Fix_P(T)$. Then we have an isomorphism $\varphi_{\delta}:A_T\to P_T$, which sends $a$ to $t_a(\delta)$. Since $D$ is a symmetric divisor, the line bundle $\oo(D)$ is symmetric, i.e. there exists a canonical isomorphism $\oo(D)=\oo(i^{-1}(D))\cong i^*\oo(D)$ of line bundles on $P$. With abuse of notation, we will denote with the same symbol the pull-back on $P_T$ of the line bundle $\oo(D)$. The map $\varphi_{\delta}$ commute with the action of the involution, then we have a canonical isomorphism $\varphi_{\delta}^*\oo(D)\cong \varphi^*_{\delta}i^*\oo(D)\cong i^*\varphi^*_{\delta}\oo(D)$. This allows us to define a morphism of $T$-schemes
$$
\begin{array}{rcl}
c_{\delta}:A_T[2]&\longrightarrow&\mu_{2,T}\\
a &\longmapsto & c_{\delta}(a):=e^{\varphi_{\delta}^*\oo(D)}(a).
\end{array}
$$
Using the properties (i), (ii), (iii), (iv) of $e^{}$, we can see that $c_{\delta}$ is a theta characteristic. Indeed, let $a,b\in A_T[2]$, then
\begin{equation}
\begin{array}{rcl}
\textbf{e}_2^{\lambda}(a,b)&\stackrel{def}{=}&\textbf{e}_2\left(a,t_b^*\varphi_{\delta}^*\oo(D)\otimes\varphi_{\delta}^*\oo(D)^{-1}\right)\stackrel{(iv)}{=}e^{t_b^*\varphi_{\delta}^*\oo(D)\otimes\varphi_{\delta}^*\oo(D)^{-1}}(a)=\\
& \stackrel{(ii)}{=} &e^{t_b^*\varphi_{\delta}^*\oo(D)}(a)\cdot e^{\varphi_{\delta}^*\oo(D)^{-1}}(a).
\end{array}
\end{equation}
Moreover
\begin{equation}
e^{\varphi_{\delta}^*\oo(D)}(a)\cdot e^{\varphi_{\delta}^*\oo(D)^{-1}}(a)\stackrel{(ii)}{=}e^{\oo_{A_T}}(a)\stackrel{(iv)}{=}\textbf{e}_2(a,O_{A_T^{\vee}})=1_T.
\end{equation}
By (i), this implies that $e^{\varphi_{\delta}^*\oo(D)}(a)=e^{\varphi_{\delta}^*\oo(D)^{-1}}(a)$. Observe that by (iii) we have
\begin{equation}
e^{t_b^*\varphi_{\delta}^*\oo(D)}(a)=e^{\varphi_{\delta}^*\oo(D)}(a+b)\cdot e^{\varphi_{\delta}^*\oo(D)}(b).
\end{equation}
Putting all together, we see that $c_{\delta}$ is a theta characteristic:
\begin{equation}\label{final}
\textbf{e}_2^{\lambda}(a,b)=e^{\varphi_{\delta}^*\oo(D)}(a+b)\cdot e^{\varphi_{\delta}^*\oo(D)}(b)\cdot e^{\varphi_{\delta}^*\oo(D)}(a)\stackrel{def}{=}c_{\delta}(a+b)c_{\delta}(b)c_{\delta}(a),
\end{equation}
Such construction is compatible with the base changes $T'\to T$. In other words, it defines a morphism of functors
$$
\begin{array}{rcl}
\phi:Fix_P&\longrightarrow&\mt T_P\\
\delta &\longmapsto & c_{\delta}.
\end{array}
$$
Moreover, the properties (i), (ii), (iii), (iv) imply also that
$
c_{t_b(\delta)}(a)=\textbf{e}_2^{\lambda}(b,a)c_{\delta}(a)
$
for any $a,b\in A_T[2]$, or, in other words, that $\phi$ is an isomorphism of $A[2]$-torsors.

Indeed, by definition $c_{t_b(\delta)}(a)=e^{\varphi_{t_b(\delta)}^*\oo(D)}(a)$ for any $a,b\in A_T[2]$. Observe that $\varphi_{t_b(\delta)}=\varphi_\delta\circ t_b$ (by an abuse of notation we have denoted with the same symbols the translation by $b\in A$ on $P$ and on $A$). In particular
\begin{equation}
e^{\varphi_{t_b(\delta)}^*\oo(D)}(a)=e^{t_b^*\varphi_{\delta}^*\oo(D)}(a)\stackrel{(iii)}{=}e^{\varphi_{\delta}^*\oo(D)}(a+b)\cdot e^{\varphi_{\delta}^*\oo(D)}(b)\stackrel{def}{=}c_{\delta}(a+b)c_{\delta}(b)
\end{equation}
On the other hand,
\begin{equation}
\textbf{e}_2^{\lambda}(b,a)c_{\delta}(a)=c_{\delta}(a+b)c_{\delta}(b)c_{\delta}(a)c_{\delta}(a)\stackrel{(i)}=c_{\delta}(a+b)c_{\delta}(b)
\end{equation}
Showing that $\phi$ is an isomorphism of $A[2]$-torsors. The second assertion follows from the fact that for any morphism $(f,g,h):(A\curvearrowright P\to S, D)\rightarrow (A'\curvearrowright P'\to S', D')$ we have $\varphi_{g(\delta)}\circ f=g\circ \varphi_\delta$.
\end{dimo}\\
 
\begin{Oss}Let $(A\curvearrowright P\to S=\op{Spec}(k), D,\varphi)$ be a geometric point in $\ms A_{g,n}$. For any $k$-point $\delta$ of $Fix_P$, let $m(\delta)$ the multiplicity of the divisor $D$ at $\delta$. By \cite[Proposition 2]{Mum66}, we have $c_{\delta}(a)=(-1)^{m(t_a(\delta))-m(\delta)}$. 
\end{Oss}

The Proposition \ref{thetacar2} allows us to prove Theorem \ref{polex}.\\\\
\begin{dimo} \emph{of Theorem \ref{polex}.}\vspace{0.2cm}

$(\Leftarrow)$. As observed before, there exists a symmetric line bundle on $\ms X_{g,n}$ inducing the universal polarization if and only if the morphism of stacks $\ms N_{g,n}\to\ms A_{g,n}$ has a section. By Proposition \ref{thetacar2}, the existence of such a section is equivalent to showing that there exists a universal theta characteristic $c:\ms X_{g,n}[2]\to\mu_{2,\ms A_{g,n}}$. Since $n$ is even, the level $n$-structure $\varPhi$ induces a level $2$-structure $\widetilde\varPhi:\ms X_{g,n}[2]\to(\bb Z/2\bb Z)^{2g}_{\ms A_{g,n}}$. Let $\text{e}:(\bb Z/2\bb Z)_{\ms A_{g,n}}^{2g}\times (\bb Z/2\bb Z)_{\ms A_{g,n}}^{2g}\to\mu_{2,\ms A_{g,n}}$ the standard symplectic pairing and $\pi_1$ (resp. $\pi_2$) the endomorphism which sends $(x',x'')\in (\bb Z/2\bb Z)^{2g}$ to $(x',0)$ (resp. to $(0,x'')$). For any $a\in \ms X_{g,n}[2](S)$, consider the map $a\mapsto c(a):= \text{e}(\pi_1\circ\widetilde\varPhi(a),\pi_2\circ\widetilde\varPhi(a))$. If we show that $c$ is a theta characteristic, we have done. More precisely, we have to show that for any $a,b\in X_{g,n}[2]$ the morphism $c$ satisfies the equality
\begin{equation}\label{ff}
\textbf{e}^{\Lambda}(a,b)=c(a+b)c(a)c(b)
\end{equation}
Indeed, on the left hand side we have
\begin{equation}
\begin{array}{rcl}
\textbf{e}^{\Lambda}(a,b)&=&\text{e}(\widetilde\varPhi(a),\widetilde\varPhi(b))=\text{e}(\pi_1\circ\widetilde\varPhi(a)+\pi_2\circ\widetilde\varPhi(a),\pi_1\circ\widetilde\varPhi(b)+\pi_2\circ\widetilde\varPhi(b))=\\
&=&\text{e}(\pi_1\circ\widetilde\varPhi(a),\pi_1\circ\widetilde\varPhi(b)+\pi_2\circ\widetilde\varPhi(b))\cdot \text{e}(\pi_2\circ\widetilde\varPhi(a),\pi_1\circ\widetilde\varPhi(b)+\pi_2\circ\widetilde\varPhi(b))=\\
&=&\text{e}(\pi_1\circ\widetilde\varPhi(a),\pi_2\circ\widetilde\varPhi(b))\cdot e(\pi_2\circ\widetilde\varPhi(a),\pi_1\circ\widetilde\varPhi(b)).
\end{array}
\end{equation}
Instead, on the right hand side
\begin{equation}
\begin{array}{rcl}
c(a+b)c(a)c(b)&=&\text{e}(\pi_1\circ\widetilde\varPhi(a+b),\pi_2\circ\widetilde\varPhi(a+b))\cdot c(a)c(b)=\\
 &=&\text{e}(\pi_1\circ\widetilde\varPhi(a),\pi_2\circ\widetilde\varPhi(a+b))\cdot \text{e}(\pi_1\circ\widetilde\varPhi(b),\pi_2\circ\widetilde\varPhi(a+b))\cdot c(a)c(b)=\\
 & =&\left(c(a)\cdot \text{e}(\pi_1\circ\widetilde\varPhi(a),\pi_2\circ\widetilde\varPhi(b))\right)\cdot\left(\text{e}(\pi_1\circ\widetilde\varPhi(b),\pi_2\circ\widetilde\varPhi(a))\cdot c(b)\right)\cdot c(a)c(b)
\end{array}
\end{equation}
Using these two equalities and the fact that we are working on $\mu_2$, the condition (\ref{ff}) follows immediately.\\

$(\Rightarrow)$ Fix $n>0$. Suppose that there exists a line bundle $\mt L$ on $\pi:\ms X_{g,n}\to\ms A_{g,n}$ inducing the universal polarization. Up to tensoring with a line bundle from $\ms A_{g,n}$, we can suppose that $\mt L$ is trivial along the zero section.\\ 
\textbf{Claim}: The line bundle $(\pi_*\mt L)^{-1}$ is, up to torsion, a square-root of the Hodge line bundle in the Picard group of $\ms A_{g,n}$.

Suppose that the claim holds. If $n$ is odd and $g\geq 4$ the Hodge line bundle does not admit a square root in $\Pic(\ms A_{g,n})$ modulo torsion (see \cite[Theorem E and Theorem 5.4]{Pu2012}). Therefore $n$ must be even.

It remains to prove the claim. Consider the cartesian diagram
$$
\xymatrix{
\ms X_{g,4n}\ar[r]^\phi\ar[d]^{\pi'} &\ms X_{g,n}\ar[d]^\pi\\
\ms A_{g,4n}\ar[r]^{\phi'} &\ms A_{g,n}
}
$$
By what we have already proved before, on $\ms X_{g,4n}$ there exists a rigidified symmetric line bundle $\mt M$ inducing the universal polarization. By \cite[Theorem 5.1, p.25]{FC90}, $\pi'_*\mt M$ is a line bundle such that 
\begin{equation}\label{fc}
(\pi'_*\mt M)^8=\left(\pi'_{*}(\Omega_{\pi'})\right)^{-4}\in \Pic(\ms A_{g,4n}).
\end{equation}
In particular $(\pi'_*\mt M)^{-1}$ is, up to torsion, a square-root of the Hodge line bundle. By Corollary \ref{incl}, we have that $\phi^*\mt L\otimes \mt P=\mt M$, where $\mt P$ is a rigidified $4n$-root line bundle. Since $\mt L$ (resp. $\mt M$) is relative ample over $\ms A_{g,n}$ (resp. $\ms A_{g,4n}$), we have that $\pi_!\mt L=\pi_*\mt L$ (resp. $\pi'_!\mt M=\pi'_*\mt M$) (see \cite[Prop. 6.13(i), p. 123]{Mum94}. Applying the Grothendieck-Riemann-Roch theorem to the morphism $\pi'$, we have the following equalities (in the rational Chow group of divisors of $\ms A_{g,4n}$)
\begin{equation}\label{GRR}
\begin{array}{rl}
c_1(\pi'_*\mt M)&=c_1(\pi'_*(\phi^*\mt L\otimes \mt P))=[\op{ch}(\pi'_*(\phi^*\mt L\otimes \mt P))]_1=\pi'_*\left(\left[\op{ch}(\phi^*\mt L)\op{ch}(\mt P)\op{Td}(\Omega^{\vee}_{\pi'})\right]_{g+1}\right)=\\
&=\pi'_*\left(\sum_{k=0}^{g+1}\frac{c_1(\mt P)^k}{k!}\left[\op{ch}(\phi^*\mt L)\op{Td}(\Omega^{\vee}_{\pi'})\right]_{g+1-k}\right)=\\
&=\pi'_*\left(\left[\op{ch}(\phi^*\mt L)\op{Td}(\Omega^{\vee}_{\pi'})\right]_{g+1}\right)= [\op{ch}(\pi'(\phi^*\mt L))]_1=c_1(\pi'_*\phi^*\mt L).
\end{array}
\end{equation}
The equality between the second and third row follows from the fact that $c_1(\mt P)^k$ (for $k\neq 0$) is a torsion element. Since $\ms A_{g,4n}$ is a smooth variety, the first Chern class map $c_1:\Pic(\ms A_{g,4n})\to \op{CH}^1(\ms A_{g,4n})$ is an isomorphism. This fact together with (\ref{fc}) and (\ref{GRR}) implies that
$(\pi_*'\phi^*\mt L)^{-1}=(\phi'^*\pi_*\mt L)^{-1}$ is, up to torsion, a root of the Hodge line bundle in $\Pic(\ms A_{g,4n})$. Since the homomorphism $\phi'^*:\Pic(\ms A_{g,n})/\Pic(\ms A_{g,n})_{\text{Tors}}\to\Pic(\ms A_{g,4n})/\Pic(\ms A_{g,4n})_{\text{Tors}}$ is injective, the claim follows.\\
\end{dimo}

\end{document}